\documentclass[11pt,letterpaper]{amsart} 
\usepackage{amssymb,amsthm, enumerate,mathtools,microtype} 
\usepackage[backref]{hyperref}

\usepackage[abbrev, nobysame]{amsrefs} 
\makeatletter
\def\bibtex@style{amsrn-mod}
\makeatother

\renewcommand{\MR}[1]{} 
\renewcommand{\PrintDOI}[1]{}
\BibSpec{arxiv}{%
    +{}  {\PrintAuthors}                {author}
    +{,} { \textit}                     {title}
    +{:} { \textit}                     {subtitle}
    +{,} { \PrintContributions}         {contribution}
    +{.} { \PrintPartials}              {partial}
    +{,} { }                            {journal}
    +{}  { \PrintDatePV}                {date}
     +{,} { \eprint}        {url}
}

\usepackage{tikz-cd}

\newtheorem{Theorem}{Theorem}[section]
\newtheorem{Lemma}[Theorem]{Lemma}
\newtheorem{corollary}[Theorem]{Corollary} 
\newtheorem{lemma}[Theorem]{Lemma}
\newtheorem{proposition}[Theorem]{Proposition}

\theoremstyle{definition} 
\newtheorem{remark}[Theorem]{Remark}

\newtheorem*{Remark}{Remark} 

 \newtheorem*{Acknowledgements}{Acknowledgements}
 
\newtheorem*{Example}{Example}

\newcommand{\cn}{\mathcal {N}} 
 
\newcommand{\cP}{\mathcal {P}} 
\newcommand{\cS}{\mathcal {S}} 
\newcommand{\cu}{\mathcal {U}} 
\newcommand{\nn}{{\mathbb N}} 
\renewcommand{\qq }{{\mathbb Q}} 
\newcommand{\rr}{{\mathbb R}} 
\newcommand{\zz}{{\mathbb Z}} 
\newcommand{\Fp}{{\mathbb F_p}} 

\newcommand{\detg}{\operatorname{det_G}} 
\newcommand{\tor}{\operatorname{tor}}
\newcommand{\rk}{\operatorname{rk}}
\DeclareMathOperator{\logtor}{logtor}
\newcommand{\Wh}{\operatorname{Wh}} 
\newcommand{\coker}{\operatorname{coker}} 

\DeclarePairedDelimiter\abs{\lvert}{\rvert}
\DeclarePairedDelimiter\floor{\lfloor}{\rfloor}
\DeclarePairedDelimiter\norm{\lVert}{\rVert} 
\newcommand{\M}[4]{ #2 \cup_{#3} M(#1)}   
\DeclareMathOperator{\rchi}{ \bar\chi} 
\renewcommand{\gg}{\gamma} 
\newcommand{\gr}{\rho} 
\newcommand{\gs}{\sigma} 
\newcommand{\gt}{\tau} 
\newcommand{\gl}{\lambda} 
\newcommand{\gG}{\Gamma} 
\newcommand{\gS}{\Sigma} 

\renewcommand{\d}{\partial} 
\newcommand{\id}{\operatorname{id}} 
\renewcommand{\ker}{\operatorname{Ker}} 

\newcommand{\im}{\operatorname{Im}} 
\newcommand{\Lk}{\operatorname{Lk}} 
\newcommand{\vol}{\operatorname{Vol}} 
\newcommand{\conv}{\operatorname{Conv}} 
\newcommand{\supp}{\operatorname{Supp}} 
\newcommand{\thi}{\operatorname{th}}

\newenvironment{enumeratei}{ 
\begin{enumerate}[\upshape (i)]}	
	{ 
\end{enumerate}
}

\numberwithin{equation}{section} 
\begin{document}\title{Torsion invariants of complexes of groups}

\author{Boris Okun} \address{Department of Mathematical Sciences, University of Wisconsin--Milwaukee, Milwaukee, WI 53201},
\email{okun@uwm.edu}

\author{Kevin Schreve} \address{Department of Mathematics, Louisiana State University, Baton Rouge, LA~70806}
\email{kschreve@lsu.edu}
\date{\today}
\begin{abstract}
    Suppose a residually finite group $G$ acts cocompactly on a contractible complex with strict fundamental domain $Q$, where the stabilizers are either trivial or have normal $\zz$-subgroups.
    Let $\d Q$ be the subcomplex of $Q$ with nontrivial stabilizers.
    Our main result is a computation of the homology torsion growth of a chain of finite index normal subgroups of $G$.
    We show that independent of the chain, the normalized torsion limits to the torsion of $\d Q$, shifted a degree.
    Under milder assumptions of acyclicity of nontrivial stabilizers, we show similar formulas for the mod $p$-homology growth.
    We also obtain formulas for the universal and the usual $L^2$-torsion of $G$ in terms of the torsion of stabilizers and topology of $\d Q$.
    In particular, we get complete answers for right-angled Artin groups, which shows they satisfy a torsion analogue of L\"uck approximation theorem.
\end{abstract}


\maketitle

\section{Introduction} Let $G$ be a residually finite group of type $F$, and let $\{\gG_{k}\}_{k \in \nn}$ be a nested chain of finite index, normal subgroups of $G$ with $\bigcap_{k} \gG_{k} = 1$.
In this paper, we are interested in the normalized growth of the homology invariants:
\[
b_{i}^{(2)}(G; \Fp)= \limsup_{k} \frac{b_{i}(B\gG_{k}; \Fp)}{[G:\gG_{k}]}, \quad t_{i}^{(2)}(G)=\limsup_{k} \frac{\logtor{H_{i}(B\gG_{k})}}{[G:\gG_{k}]}.
\]
Here $b_{i}(B\gG_{k}; \Fp)=\dim_{\Fp} H_{i}(B\gG_{k}; \Fp)$ denotes the $i^{th}$ Betti number with coefficients in a field $\Fp$, and $\logtor H_{i}(B\gG_{k})$ denotes the logarithm of the order of the torsion subgroup $\tor H_{i}(B\gG_{k})$ of the integral homology.

We will call the first quantity the $i^{th}$ $\Fp$-$L^2$-Betti number of $G$, and the second the \emph{$i^{th}$ torsion growth} of $G$.
By L\"uck's approximation theorem \cite{l94}, if instead of $\Fp$ we take rational coefficients, the first quantity coincides with the $i^{th}$ $L^2$-Betti number of $G$, $b_{i}^{(2)}(G)$, and therefore is an honest limit and does not depend on the choice of normal chain.
Neither of the these two properties is known for the quantities above.

If $G$ is $L^2$-acyclic (has $b_{i}^{(2)}(G) = 0$ for all $i$), one can define a secondary invariant called the \emph{$L^2$-torsion} of $G$, denoted by $\gr^{(2)}(G)$.
A conjectural version~\cite{l13}*{Conjecture 1.11} of L\"uck's approximation theorem for $t_{i}^{(2)}(G)$ states that
\[
\sum_{i} (-1)^it_{i}^{(2)}(G) = \gr^{(2)}(G).
\]
Apart from several vanishing results, cf.~\cites{l13,s16,abfg21,kkn17,agn17}, very little is known about $t_{i}^{2}(G)$.
For example, if $M^3$ is a closed hyperbolic $3$-manifold, the conjecture predicts that $t_1^{(2)}(\pi_1(M^3)) = \frac{1}{6\pi}\vol(M^3)$.
In this case, L\^e~\cite{l18} proved that $t_1^{(2)}(\pi_1(M^3)) \le \frac{1}{6\pi}\vol(M^3)$, but there is no example of any aspherical $3$-manifold where it is known that $t_1^{(2)}(\pi_1(M^3)) > 0$.

Note that the universal coefficients theorem gives a lower bound
\[
t_{i}^{(2)}(G) + t_{i-1}^{(2)}(G) \geq \left(b_{i}^{(2)}(G; \Fp) - b_{i}^{(2)}(G)\right)\log p.
\]
With Avramidi we exploited this in \cite{aos21} to give examples of nontrivial torsion growth, by computing the $\Fp$-$L^2$-Betti numbers of right-angled Artin groups (RAAG's.) A new feature in this paper is an exact calculation of nonvanishing torsion growth and verification of the approximation conjecture for RAAG's.

In fact, we work in greater generality.
The natural setting for such calculations is group actions with strict fundamental domains, with suitable assumptions on stabilizers.
Recall that a strict fundamental domain for a cellular action on a complex is a subcomplex which intersects each orbit in a single point.
In particular, the quotient by the group action is isomorphic to the strict fundamental domain; hence it is a finite complex if the action is cocompact.
A general construction of actions with strict fundamental domain comes from simple complexes of groups, as in \cite{bh99}*{Chapter II.12}.
For some examples to keep in mind, note that a Euclidean triangle group acts on $\rr^2$ with strict fundamental domain, whereas there is no strict fundamental domain for the standard action of $\zz^2$ on $\rr^2$.

\subsection*{Homology growth} In Section~\ref{s:computation}, using an equivariant homology spectral sequence, we calculate the $\Fp$-$L^2$-Betti numbers of groups acting on a contractible complex with strict fundamental domain $Q$ and $\Fp$-$L^2$-acyclic stabilizers.
Our main theorem is a similar calculation for $t_{i}^{(2)}(G)$, though we need stronger assumptions on the stabilizers.
\begin{Theorem}\label{t:main}
    Let $G$ be a residually finite group which acts cocompactly on a contractible complex with strict fundamental domain $Q$.
    Suppose the stabilizer of any cell fixes it, and that each nontrivial stabilizer has a normal, infinite cyclic subgroup with type $F$ quotient.
    Let $\d Q$ be the subcomplex of $Q$ with nontrivial stabilizers.
    Then
    \[
    t_{i}^{(2)}(G)=\logtor H_{i-1}(\d Q).
    \]
\end{Theorem}

As with the computation in \cite{aos21}, the $\limsup$ is an honest limit and is independent of the chain.
Our main class of groups which satisfy the assumptions in Theorem \ref{t:main} are residually finite Artin groups which satisfy the $K(\pi,1)$-conjecture, in particular RAAG's.
Associated to such an Artin group $A$ is a simplicial complex $L$ called the \emph{nerve} whose simplices correspond to special Artin subgroups of finite type, and $A$ acts cocompactly on a contractible complex called the Deligne complex with strict fundamental domain isomorphic to the cone on $L$ ~\cite{cd95}*{Section 1.5}.

The stabilizers of simplices are either trivial or isomorphic to Artin groups of finite type, which have normal infinite cyclic subgroups, and the simplices in $L$ correspond precisely to simplices in the Deligne complex with nontrivial stabilizer.
Therefore, we have that $t_{i}^{(2)}(A) = \logtor H_{i-1}(L)$.

For an explicit example, if $L$ is a flag triangulation of $\rr P^2$, then our earlier work with Avramidi showed that the RAAG $A_L$ had $t_2^{(2)}(A_L) > 0$, and the work in this paper shows that $t_2^{(2)}(A_L) = \log{2}.$

Our proof of Theorem \ref{t:main} uses the recent work of Abert, Bergeron, Fraczyk, and Gaboriau in \cite{abfg21}.
They developed a general strategy for showing the vanishing of $t_{i}^{(2)}(G)$ for groups which act on contractible complexes with ``cheap" infinite stabilizers.
Here, ``cheap" roughly means that finite index subgroups admit classifying spaces with sublinear (in the index) number of cells and subexponential norm of boundary maps, see Section 10 of \cite{abfg21}.
By using an effective version of Geoghegan's rebuilding procedure for the Borel construction \cite{g08}*{Section 6.1}, they build classifying spaces for finite index subgroups of $G$ by gluing together these nice classifying spaces of the stabilizers, and show the resulting spaces have vanishing torsion growth.

Our argument has two parts.
The first part is essentially handled by the method in \cite{abfg21}.
In Section \ref{s:rebuilding}, we describe an alternative approach to the effective rebuilding procedure using iterated mapping cylinders, which we find simpler.
Given a group $G$ as in Theorem \ref{t:main}, we construct a classifying space $BG$ built out of classifying spaces of the stabilizers $BG^\gs$.
There is a subcomplex $Y$ which is built from $BG^\gs$ for $\gs$ in $\d Q$, and $BG = Q \cup_{\d Q} Y$.
Since groups with normal infinite cyclic subgroups and type $F$ quotient are cheap, the lifts $Y_{k}$ of $Y$ to the finite cover $B\gG_{k}$ have vanishing torsion growth.

The second part of the argument deals with the lifts of $Q$ inside $B\gG_{k}$.
There are a linear number of such lifts (note that $Q$ is contractible), and they are all glued along lifts of $\d Q$ inside $Y_{k}$.
Therefore, we have a long exact sequence relating the homology of $B\gG_{k}$, the homology of $\d Q$, and the homology of $Y_{k}$.
The difficulty here is that torsion does not play very nicely with exact sequences; unlike Betti numbers, small torsion of a term is not implied by small torsion of its neighbors.
The fact that the norm of the attaching maps of $Q$ are bounded again by a polynomial independent of the cover, and that the homology of $\d Q$ has an upper bound on its torsion subgroups, will imply that the torsion in this long exact sequence behaves as naively expected, which implies the theorem.

\subsection*{\texorpdfstring{$L^2$}{L2}-torsion} For many $L^2$-acyclic groups, Friedl and L\"uck \cite{fl17} have defined an algebraic generalization of $\gr^{(2)}(G)$ called the \emph{universal $L^2$-torsion} $\gt_u^{(2)}(G)$.
We describe this here in a special case.
If $G$ is torsion-free and satisfies the Atiyah conjecture on integrality of $L^2$-Betti numbers, Linnell~\cite{l93} showed there is a certain skew field $D(G)$ containing $\zz G$, and $b_{i}^{(2)}(G) = \dim_{D(G)}H_{i}(G; D(G))$.
Under some additional assumptions on $G$ (satisfied by RAAG's for instance), the universal $L^2$-torsion can be identified with the Reidemeister torsion of $G$ with coefficients in $D(G)$.
This lives in the Whitehead group $\Wh(D(G))$ and determines the usual $L^2$-torsion $\gr^{(2)}(G)$, but contains more information.
For example, Friedl and L\"uck~\cite{fl19} showed that it determines a convex polytope in $H_1(G; \rr)$, and for $\phi \in H^1(G; \rr)$, the thickness of the polytope in the $\phi$-direction is precisely the $L^2$-Euler characteristic of $\ker \phi$.
If $M^3$ is a closed, aspherical $3$-manifold, then it follows from \cite{ll95} and Perelman's proof of Thurston geometrization that $\pi_1(M^3)$ is $L^2$-acyclic; the polytope in this case is essentially dual to the unit ball of the Thurston norm.

If a group $G$ acts cocompactly on a contractible complex with strict fundamental domain $Q$ and the stabilizers are either trivial or $L^2$-acyclic, then $G$ is $L^2$-acyclic if and only if $\partial Q$ is $\qq$-acyclic.
In this case, we can calculate $\gt_u^{(2)}(G)$ and $\gr^{(2)}(G)$.
We refer to Theorem \ref{t:l2torsion1} for the precise statement.
For RAAG's based on a $\qq$-acyclic flag complex $L$, the upshot is that the universal $L^2$-torsion detects two features of the topology of $L$.
The first is the product of the torsion groups $\prod \abs{\tor H_{i}(L)}^{(-1)^i}$, and the second is the (reduced) Euler characteristic of links of vertices of $L$.
This second term could have been predicted by a formula of Davis and the first author \cite{do12} for $L^2$-Betti numbers of the Bestvina--Brady group (the kernel of the standard homomorphism $A_L \rightarrow \zz$ which sends each generator to $1$):
\[
b_{i}^{(2)}(BB_{L}) = \sum_{s \in L^{(0)}} \bar b_{i-1}(\Lk(s); \qq ).
\]

However, the second term does not contribute to the usual $L^{2}$-torsion $\gr^{(2)}(A_{L})$, and combining this with Theorem \ref{t:main} shows that
\[
\sum_{i} (-1)^i t_{i}^{(2)}(A_L) = \gr^{(2)}(A_L),
\]
so the L\"uck torsion approximation conjecture holds for RAAG's.
The proof works for any group satisfying the hypothesis of Theorem \ref{t:main} and some technical assumptions, see Corollary~\ref{c:nz}), though in general there is not as clean a description of the universal $L^2$-torsion.

\subsection*{Disclaimer} For simplicity, we will only work with groups of type $F$.
In particular, in the remainder of the paper, all \emph{groups, subgroups} and \emph{quotients of normal subgroups} will be assumed to be type $F$.
We will say explicitly at certain points when we do not need this strong an assumption.

\subsection*{Notation} A cellular action is called \emph{rigid} if the action of the stabilizer of any cell fixes it.
Most of our theorems start with the assumption that there is a group $G$ which acts cocompactly and rigidly on a contractible complex $Z$ with strict fundamental domain $Q$.
Note that we can always take a barycentric subdivision to ensure that the action is rigid.
The cell stabilizers are denoted by $G^{\gs}$ and, furthermore, $\d Q$ denotes the subcomplex of $Q$ with nontrivial stabilizers.
In these theorems, we will shorten this to ``A group $G$ acts on a contractible complex $Z$ with strict fundamental domain $(Q, \d Q)$".
\begin{Acknowledgements}
    We thank Mikolaj Fraczyk and Damien Gaboriau for reading an earlier version of this article and making valuable comments.
    We thank Grigori Avramidi for many helpful conversations.
    We also thank the anonymous referees whose careful readings and comments have led to substantial improvements in the exposition of this paper.
    The first author would like to thank the Max Planck Institute for Mathematics for its support and excellent working conditions.
    This material is based upon work done while the second author was supported by the National Science Foundation under Award No.~1704364 and under the grant DMS-2203325. 
\end{Acknowledgements}

\section{Rebuilding mapping cylinders}\label{s:rebuilding}

In this section, all spaces will be finite CW-complexes.
We abuse notation to improve readability.
Given a cellular map $f: X \to Y$, we will denote by $f$ the induced map on $i$-chains for any $i$.
We will also not distinguish between various differentials in a chain complex, they will all be denoted by $\partial$ or $\partial_{\text{space}}$ when we want to remember the space.
By the norm of a linear map $\rr^n \to \rr^m$ we always mean the operator norm with respect to the Euclidean norm.
The norm on chains $C_\ast$ will always come from the standard inner product where the cells form an orthonormal basis.

Now, let $X$ be a CW-complex and let $(F, E)$ be a CW-pair.
Given a cellular map $f: E \to X$ we can attach $F$ to the mapping cylinder $M(f)$ to obtain $ \M{f}{F}{E}{X} $.
\begin{Lemma}\label{l:diff}
    The norm of the differential in the chain complex of $ \M{f}{F}{E}{X}$ satisfies
    \[
    \norm{\d} \le 1 + \norm{f} + \norm{\d_{F}} + \norm{\d_{E}} + \norm{\d_{X}}.
    \]
\end{Lemma}
\begin{proof}
    The chain complex of $ \M{f}{F}{E}{X} $ splits orthogonally as
    \[
    C_\ast( \M{f}{F}{E}{X} ) \cong C_\ast(F) \oplus C_{\ast-1}(E) \oplus C_\ast(X).
    \]
    The matrix of the differential $\d$ has corresponding block form:
    \[
    \begin{pmatrix}
        \d_{F} & i & 0 \\
        0 & -\d_{E} & 0 \\
        0 & -f & \d_{X}
    \end{pmatrix}
    \]
    where $i: E \rightarrow F$ is the inclusion map.
    Since the norm of a matrix is bounded by the sum of the norms of its blocks, the claim follows.
\end{proof}

Now suppose we are given homotopy inverses $h: X \to X'$, $h' : X' \to X$, as well as homotopy inverses $g: (F,E) \to (F',E')$ and $g': (F',E') \to (F,E)$.
Let $\gs: X \times I \to X$ be the homotopy between the identity and $h'h$ and $\gg : (F,E) \times I \to (F,E)$ be the homotopy between the identity on $(F,E)$ and $g' g$.

We want to construct a homotopy equivalent version of $\M{f}{F}{E}{X}$ using $X', F'$, and $E'$ and have some control on the norms of the homotopy equivalences and the new boundary operator.
To this purpose, define $f': E' \to X'$ by $f'=h f g'$.
Clearly,
\begin{equation}\label{e:f}
    \norm{f'} \le \norm{h} \norm{f} \norm{g'}.
\end{equation}
\begin{Lemma}\label{i:H}
    There exist homotopy inverses $H: \M{f}{F}{E}{X} \to \M{f'}{F'}{E'}{X'}$ and $H': \M{f'}{F'}{E'}{X'} \to \M{f}{F}{E}{X}$ and a homotopy $\gS$ between $H'H$ and the identity on $\M{f}{F}{E}{X}$ satisfying
    \begin{align*}
        \norm{H} &\le \norm{ g } +\norm{ g_{E} }+\norm{ h } + \norm{h} \norm{f} \norm{\gg}, \\
        \norm{H'} &\le \norm{ g' } +\norm{ g'_{E} }+\norm{ h' } + \norm{\gs} \norm{f} \norm{g'_E}, \\
        \norm{\gS} &\le \norm{ \gg } +\norm{ \gs}+\norm{ \gg_E} + \norm{\gs} \norm{f} \norm{\gg}.
    \end{align*}
\end{Lemma}
\begin{proof}
    There is a commutative diagram:
    \[
    \begin{tikzcd}[row sep=large, column sep = large] F \arrow{d}{g} & \arrow[hook']{l}[swap]{i} E \arrow{d}{g_E} \arrow{r}{fg'g} & X \arrow{d}{h}\\
        F' & \arrow[hook']{l}[swap]{i} E' \arrow{r}{f'} & X'
    \end{tikzcd}
    \]
    The maps $g|_E \times \id: E \times I \to E' \times I$, $g: F \to F'$, and $h: X \to X'$ induce a map between mapping cylinders:
    \[
    g \cup h: \M{fg'g}{F}{E}{X} \to \M{f'}{F'}{E'}{X'}.
    \]
    By \cite{g08}*{Theorem 4.1.5}, $g \cup h$ is a homotopy equivalence.
    Since $g'g$ is homotopic to the identity via $\gg$, $ \M{f}{F}{E}{X} $ is homotopy equivalent to $ \M{fg'g}{F}{E}{X} $.
    The homotopy equivalence
    \[
    \Phi: \M{f}{F}{E}{X} \to \M{fg'g}{F}{E}{X}
    \]
    is given by
    \[
    \Phi(x,t) =
    \begin{cases}
        (x,2t) & 0 \le t \le 1/2 ,\\
        f(\gg(x,2(1-t))) & 1/2 \le t \le 1.
    \end{cases}
    \]
    on the $E \times I$, and the identity maps on $F$ and $X$.

    The composition of these homotopy equivalences gives a homotopy equivalence
    \[
    H: \M{f}{F}{E}{X} \xrightarrow{\Phi} \M{fg'g}{F}{E}{X} \xrightarrow{g \cup h} \M{f'}{F'}{E'}{X'} .
    \]
    To estimate its norm, we use splittings of the chain complexes of $ \M{f'}{F'}{E'}{X'} $ and $ \M{fg'g}{F}{E}{X} $, similar to Lemma~\ref{l:diff}.
    The map on chains induced by $g \cup h$ is block diagonal $g \cup h = g \oplus g_{E} \oplus h$.
    The matrix of $\Phi$ has block form:
    \[
    \begin{pmatrix}
        id & 0 & 0 \\
        0 & id & 0 \\
        0 & -f \gg & id
    \end{pmatrix}
    \]
    Thus, $H=(g \cup h) \Phi $ has matrix:
    \[
    \begin{pmatrix}
        g & 0 & 0 \\
        0 & g_{E} & 0 \\
        0 & -hf \gg & h
    \end{pmatrix}
    \]
    and the claim follows.

    The proof for $H'$ is similar.
    Note that there is another commutative diagram:
    \[
    \begin{tikzcd}[row sep=large, column sep = large] F' \arrow{d}{g'} & \arrow[hook']{l}[swap]{i} E' \arrow{d}{g'_E} \arrow{r}{f'} & X' \arrow{d}{h'}\\
        F & \arrow[hook']{l}[swap]{i} E \arrow{r}{h'hf} & X
    \end{tikzcd}
    \]
    Therefore, the induced map $g' \cup h': \M{f'}{F'}{E'}{X'} \to \M{h'hf}{F}{E}{X}$ is a homotopy equivalence.
    The mapping cylinders $\M{h'hf}{F}{E}{X}$ and $\M{f}{F}{E}{X}$ are homotopy equivalent through the explicit homotopy
    \[
    \Psi(x,t) =
    \begin{cases}
        (x,2t) & 0 \le t \le 1/2, \\
        \gs(f(x), 2t-1) & 1/2 \le t \le 1.
    \end{cases}
    \]
    on $E \times I$ and the identity on $F$ and $X$.
    Thus, the matrix for $H' = \Psi \circ (g' \cup h')$ is
    \[
    \begin{pmatrix}
        id & 0 & 0 \\
        0 & id & 0 \\
        0 & \gs f & id
    \end{pmatrix}
    \begin{pmatrix}
        g' & 0 & 0 \\
        0 & g'_{E} & 0 \\
        0 & 0 & h'
    \end{pmatrix}
    =
    \begin{pmatrix}
        g' & 0 & 0 \\
        0 & g'_{E} & 0 \\
        0 & \gs f g'_E & h'
    \end{pmatrix}
    \]

    There are explicit though somewhat horrendous formulas for the composition $H'H$ and the homotopy $\Sigma$.
    On the other hand, we only need to estimate the norm of the induced chain homotopy, so at this point will just observe that the matrix
    \[
    \begin{pmatrix}
        \gg & 0 & 0 \\
        0 & -\gg_E & 0 \\
        0 & -\gs f \gg_E & \gs
    \end{pmatrix}
    \]
    gives a chain homotopy $\gS$ between $H'H$ and the identity (which the reader can easily verify).
\end{proof}
\begin{Remark}
    We shall refer to the starting maps $g$, $g'$, $h$, $h'$ and $\gs$ as the \emph{rebuilding maps}, and $H$, $H'$, $\gS$ as the \emph{rebuilt maps}.
    We will often be inductively going through this rebuilding procedure, in which case the rebuilt maps will become the next stage's rebuilding maps (along with new maps $g$ and $g'$ from a new pair $(F,E)$.)
\end{Remark}

\subsection{Filtered Complexes}

Let $X$ be a CW-complex.
A \emph{cylindrical filtration} of $X$ consists of a collection of increasing subcomplexes $\{X_{j}\}_{j = 0}^n$ with $X_{n} = X$ and CW-pairs $\{(F_{j}, E_{j}) \}_{j= 1}^n$ with maps $f_{j}: E_{j} \to X_{j-1}$ for $j > 0$ so that
\[
F_0 = X_0 \text{ and }X_{j} = \M{f_{j}}{F_{j}}{E_{j}}{X_{j-1}}.
\]

Now, suppose that we have $X$ with a cylindrical filtration of length $n$, and a collection of pairs $\{(F_{j}',E_{j}')\}_{j = 0}^n$ homotopy equivalent to the $\{(F_{j},E_{j})\}_{j = 0}^n$.
Then by applying Lemma \ref{i:H} repeatedly, we build a homotopic complex $X'$ with a cylindrical filtration, and homotopy inverses $H: X \to X'$, $H': X' \to X$, and a homotopy $\gS: X \times I \to X$ between $H'H$ and the identity as rebuilt maps.
Let $g_{j}: (F_{j}, E_{j}) \to (F_{j}', E_{j}')$, $g_{j}': (F_{j}, E_{j}') \to (F_{j}, E_{j})$ and $\gg_{j}: (F_{j}, E_{j}) \times I \to (F_{j}, E_{j})$ be the relevant starting data for our rebuilding.
The next two lemmas are our main source of control for norms of the rebuilt maps in terms of the rebuilding maps.

Let $M$ be an upper bound for the norms of the attaching maps $f_{j}$ in the initial filtration.
Suppose the norms of rebuilding maps $\norm{g_{j}}$, $\norm{g_{j}'}$, $\norm{\partial_{F'}}$ and $\norm{\gg_{j}}$ are all bounded above by $K$.
Inductive application of Lemmas \ref{l:diff} and \ref{i:H}, and equation (\ref{e:f}) gives a polynomial in $M$ and $K$ which bounds $\norm{H}, \norm{H'}, \norm{\gS}$ and $\norm{\d_{X'}}$.
Regarding it as a polynomial in $K$ gives the following:
\begin{Lemma}\label{l:control}
    There is a constant $C=C(M, n)$ such that if
    \[
    \norm{g_{j}}, \norm{g_{j}'}, \norm{\gg_{j}}, \norm{\partial_{F_{j}'}} < K
    \]
    and $K > 2$, then
    \[
    \norm{H}, \norm{H'}, \norm{\gS}, \norm{\partial_{X'}} < CK^C.
    \]
\end{Lemma}

Now let $\widehat X \to X$ be a cover of $X$.
The cylindrical filtration on $X$ induces a natural cylindrical filtration on $\widehat X$, where $\widehat X_j$ and the pair $(\widehat F_{j}, \widehat E_{j})$ are the preimages of $X_j$ and $(F_{j}, E_{j})$, and $\widehat f_{j}$ is the lift of $f_{j}$.
\begin{Lemma}\label{l:cover}
    The norm of $\widehat f_{j}$ is uniformly bounded, independent of the cover.
\end{Lemma}
\begin{proof}
    On the level of the universal cover the attaching map is described by a matrix with $\zz\pi_{1}$ coefficients.
    The uniform bound for $\norm{\widehat f_{j}}$ in terms of this matrix is given by \cite{l94}*{Lemma 2.5}.
\end{proof}

To summarize, given a cylindrical filtration on $X$, the norm of the rebuilt maps of any rebuilding of a finite cover of $X$ is bounded by a fixed polynomial of the maximal norm of the rebuilding maps.

Any regular neighborhood of a subcomplex is a mapping cylinder neighborhood as observed in \cite{kr63}.
Therefore, any filtration of $X$ by subcomplexes leads to a cylindrical filtration (on the second barycentric subdivision of $X$) by taking pairs $(F_{j}, E_{j})$ to be $(X_{j} -N(X_{j-1}), \d N(X_{j-1}))$, where $N(X_{j-1})$ is a regular neighborhood of $X_{j-1}$ in $X_{j}$.
A particularly easy case is a filtration of $X$ by its skeleta $X^{j}$, then we can take $(F_{j}, E_{j})$ to be $\sqcup (\gs^{j}, \d \gs^{j})$ with the standard attaching maps.
It is easy to change such filtration into a cylindrical filtration and keep the same pairs $(F_{j},E_{j})$ at the cost of changing the attaching maps.
\begin{Lemma}\label{l:filter}
    Suppose $X$ is filtered by subcomplexes $X_0 \subset X_1 \subset \dots \subset X_{n} = X$ where $X_{j} = F_{j} \cup_{f_{j}} X_{j-1}$ for a pair $(F_{j}, E_{j})$ and a map $f_{j}: E_{j} \to X_{j-1}$.
    Then there is a homotopy equivalent complex $X'$ with a cylindrical filtration $X_0' \subset X_1' \subset \dots \subset X_{n}' = X$ and maps $f_{j}': E_{j} \to X_{j-1}'$ so that $X_{j}' = \M{f_{j}'}{F_{j}}{E_{j}}{X_{j-1}'}$.
\end{Lemma}
\begin{proof}
    Suppose by induction that $h:X_{j-1} \to X_{j-1}'$ is a homotopy equivalence and $X_{j-1}'$ has a cylindrical filtration.
    Set $X_{j}' = \M{hf_{j}}{F_{j}}{E_{j}}{X_{j-1}'}$.
    Then $X_{j}'$ is homotopy equivalent to $X_{j}$ by Lemma \ref{i:H}.
\end{proof}

This applies to the natural filtration on a Borel construction.
Suppose $G$ acts cocompactly and rigidly on a complex $Z$.
Then $X=EG \times_{G} Z$ has a natural projection $\pi: X \to Z/G$, and therefore $X$ has a filtration by the preimages of skeleta of $Z/G$.
This filtration is an iterated adjunction space where
\[
(F_{j}, E_{j}) = \bigsqcup_{\gs \in Z^{(j)}/G} EG/G^{\gs} \times (\gs, \partial\gs),
\]
and $G^{\gs}$ denotes the stabilizer of a lift of $\gs$ to $Z$.

By Lemma~\ref{l:filter} we can replace this filtration by a cylindrical filtration and keep the same pairs $(F_{j},E_{j})$.
Furthermore, since $EG/G^{\gs}$ is a model of the classifying space of $BG^{\gs}$, we can rebuild this filtration using any other collection of models $\{ BG^{\gs}\}_{\gs \in Z/G}$.
Adding finiteness assumptions for stabilizers gives the following lemma.
\begin{Lemma}\label{l:borel}
    Suppose that $G$ acts cocompactly and rigidly on a contractible complex $Z$.
    Suppose that all stabilizers have finite $BG^\gs$.
    Then there is a finite $BG$ with a cylindrical filtration where
    \[
    (F_{j}, E_{j}) = \bigsqcup_{\gs \in Z^{(j)}/G} BG^\gs \times (\gs, \partial \gs).
    \]
\end{Lemma}
\begin{remark}
    Similar results to those above could have been obtained using the effective rebuilding procedure of Abert, Bergeron, Fraczyk, and Gaboriau in \cite{abfg21}*{Section 4 and 5}.
    They did not use mapping cylinders, but rather stacks of CW-complexes as in \cite{g08}*{Chapter 6}.
    We included this section to make this paper more self-contained and because we think the mapping cylinder approach is simpler.
\end{remark}

\section{Torsion bounds}\label{s:torsionlemma}

In the next section we will use a long exact sequence in our computation of homology torsion growth.
We will need bounds for torsion in the short exact sequences that come from this.
The following two lemmas will be useful.
\begin{lemma}\label{l:logtor}
    Let $0 \to A \xrightarrow{i} B \xrightarrow{j} C \to 0$ be a short exact sequence of finitely generated abelian groups.
    \begin{enumeratei}
        \item\label{i:subadd}
        $\logtor B \le \logtor A + \logtor C$.
        \item\label{i:add}
        If $A$ is finite, then $\logtor B = \logtor A + \logtor C$.
        \item\label{i:snf}
        If $A$ and $B$ are both free abelian, then $\logtor C \leq \rk A \log e(C)$, where $e(C)$ denotes the maximal order of torsion elements in $C$ (the exponent of $\tor C$.)
        \item\label{i:lb}
        In general, $\logtor B \geq \logtor A + \logtor C - \rk A \log e(C)$.
    \end{enumeratei}
\end{lemma}
\begin{proof}
    Subadditivity of $\logtor$ (\ref{i:subadd}) is well known and follows immediately from the left exact induced sequence of torsion subgroups.
    If $A$ is finite, then the induced sequence is exact, and we obtain (\ref{i:add}).
    To prove (\ref{i:snf}) choose bases for $A$ and $B$ so that the matrix of $i$ is diagonal.
    Then $\logtor C$ is the sum of at most $\rk A$ logs of nonzero entries, each of them bounded by $\log e(C) $.

    Finally, to prove (\ref{i:lb}) consider the following diagram with short exact rows and columns:
    \[
    \begin{tikzcd}[row sep=small, column sep = small] \tor A \ar{d} \ar{r}{i} & \tor B \ar{d} \ar{r}{j} & j(\tor B) \ar{d} \\
        A \ar{d} \ar{r}{i} & B \ar{d} \ar{r}{j} & C \ar{d} \\
        A/\tor A \ar{r}{i} & B/\tor B \ar{r}{j} & C/j(\tor B)
    \end{tikzcd}
    \]
    We can apply (\ref{i:add}) to the top row and to the right column to conclude that $$\logtor B = \logtor A + \logtor C - \logtor C/j(\tor B).$$ Then applying (\ref{i:snf}) to the bottom row implies that
    \[
    \logtor C/j(\tor B) \leq \rk A\log e(C/j(\tor B)),
    \]
    and since $j(\tor B)$ is torsion we have $e(C/j(\tor B)) \leq e(C)$.
    Combining these gives (\ref{i:lb}).
\end{proof}
\begin{lemma}\label{l:logtor1}
    Let $A \xrightarrow{f} B \xrightarrow{} C \xrightarrow{} D \xrightarrow{} E $ be an exact sequence of finitely generated abelian groups.
    Then
    \[
    0 \leq \logtor B/\im f - \logtor C + \logtor D \leq \logtor E + \rk B \log e(D).
    \]
\end{lemma}
\begin{proof}
    Denote the kernels $C':=\ker (C \to D)$ and $D' :=\ker (D \to E)$.
    We have two short exact sequences:
    \begin{gather*}
        0 \to C' \to C \to D' \to 0, \\
        0 \to D' \to D \to E' \to 0.
    \end{gather*}
    Applying Lemma \ref{l:logtor}\eqref{i:subadd} and \eqref{i:lb} to the first sequence gives
    \[
    \logtor C' + \logtor D' - \rk C' \log e(D') \leq \logtor C \leq \logtor C' + \logtor D'.
    \]

    Monotonicity of $\rk$ under surjections and of $e$ under injections implies $\rk C' \leq \rk B$ and $\log e(D') \leq \log e(D)$, so rearranging,
    \[
    -\logtor D' \leq \logtor C' - \logtor C \leq - \logtor D' + \rk B \log e(D).
    \]
    From the second sequence, using Lemma \ref{l:logtor}\eqref{i:subadd} and monotonicity of $\logtor$ under injections, we obtain
    \[
    \logtor D' \leq \logtor D \leq \logtor D' + \logtor E.
    \]
    Adding these inequalities and noting that $C'= B/\im f$ proves the claim.
\end{proof}

We shall also need the following proposition of Gabber, see \cite{abfg21}*{Proposition 9.1} for a detailed proof.
\begin{proposition}\label{p:gabber}
    Suppose that $f: \zz^n \rightarrow \zz^k$ is a $\zz$-linear map.
    Then
    \[
    \logtor \coker f \leq \rk f \log \max(\norm f , 1).
    \]
    In particular, for a finite CW-complex $X$
    \[
    \logtor H_j(X) \leq \rk C_j(X) \times \log \max(\norm{\d_{j+1} } , 1).
    \]
\end{proposition}

\section{Examples}\label{s:examples}

Our main computations of homological growth will be for groups acting on contractible complexes with strict fundamental domain.
Such actions have an equivalent description in terms of simple complexes of groups.
We review a general procedure for constructing such groups and complexes and two specific examples, see \cite{bh99}*{Chapter II.12} for full details.

A \emph{simple complex of groups} over a poset $\cP$ is a collection of groups $\{G^\gs\}_{\gs \in \cP}$ and monomorphisms $\phi_{\gs\gt}: G^\gs \to G^\gt$ so that if $\gs < \gt < \mu$ then $\phi_{\gs\mu} = \phi_{\gs\gt}\phi_{\gt\mu}$.
Let $G$ denote the direct limit of this directed system.
A simple complex of groups is \emph{strictly developable} if the canonical maps $G^{\gs} \to G$ are injective.

Now, if a group $G$ acts on a CW-complex $Z$ with strict fundamental domain $Q$, then the collection of stabilizers $G^{\gs}$ of cells in $Q$ is a strictly developable simple complex of groups (over itself partially ordered by inclusion) and, if $Z$ is connected and simply connected, then its direct limit is $G$.

In the opposite direction, given a strictly developable simple complex of groups over $\cP$, let $\abs{\cP}$ be the geometric realization of $\cP$.
Every point $x \in \abs{\cP}$ is contained in a simplex corresponding to some chain in $\cP$, and we let $\gs(x)$ denote the smallest element of such chain.
Let $\cu(G, \abs{\cP} )$ be the \emph{basic construction} associated to this data:
\[
\cu(G, \abs{\cP}) = (G \times \abs{\cP}) /\sim, \
\]
where $(g,x) \sim (g',x')$ if and only if $x = x'$ and $gG^{\gs(x)} = g'G^{\gs(x)}$.

Then $G$ acts on $\cu(G, \abs{\cP} )$ with $1 \times \abs{\cP} $ as a strict fundamental domain, which we identify with $\abs{\cP}$.
The stabilizer of a point $(g,x)$ in $\cu(G, \abs{\cP})$ is the conjugate $gG^{\gs(x)}g^{-1}$.
In particular, if $\gt$ is a simplex in $ \abs{\cP}$, then the stabilizer of $\gt$ is the group $G^{\min \gt}$, where $\min \gt$ is the smallest element in the chain representing $\gt$.

In general $\cu(G, \abs{\cP} )$ is not a contractible complex, but in many interesting cases it is.
Since $\abs{\cP}$ is a retract of $\cu(G, \abs{\cP})$, this at least requires contractibility of $ \abs{\cP} $.

A special, but common case is when the poset is the poset $\cS(L)$ of simplices of a simplicial complex $L$ ordered by inclusion (including the empty simplex, one usually assumes $G^{\emptyset}=1$.) The geometric realization of this poset is called the \emph{Davis chamber} $K$.
It is isomorphic to the cone on the barycentric subdivision of $L$.
The geometric realization of the poset of nonempty simplices $\cS_0(L)$ is isomorphic to the barycentric subdivision of $L$ and will be denoted $\partial K$.
\begin{Example}
    Let $L$ be a flag complex, and let $\{G^v\}_{v \in L^0}$ be a collection of groups.
    For a simplex $\gs \in \cS(L)$, let $G^\gs = \prod_{v \in \gs} G^v$, and if $\gs < \gt$, let $i_{\gs\gt}$ be the natural inclusion of direct products.
    The direct limit of this strictly developable simple complex of groups is the \emph{graph product} of the $\{G^v\}$.
    For example, if all vertex groups are infinite cyclic, then $G_L$ is a RAAG. For graph products, the complex $\cu(G_{L}, K)$ admits a CAT(0) cubical metric, and is hence contractible.
    Furthermore, if each vertex group is residually finite, then $G_L$ is as well \cite{g90}.
\end{Example}
\begin{Example}
    Suppose that $A$ is a (not necessarily right-angled) Artin group with nerve $L$.
    Given a simplex $\gs$ in $L$, there is an associated spherical Artin group $A^\gs$, and if $\gs < \gt$, there is a natural inclusion $A^\gs \to A^\gt$.
    Furthermore, the natural map $A^\gs \to A$ is injective \cite{l83}.
    Therefore, $A$ is the direct limit of the strictly developable simple complex of groups over $S(L)$ with local groups $A^\gs$.
    In this case, the basic construction is called the Deligne complex.
    One version of the $K(\pi,1)$-conjecture for Artin groups states that the Deligne complex is contractible, and this is known in many cases, see \cite{cd95}.
    By Deligne's theorem \cite{d72}, any spherical Artin group has a finite classifying space.
    We shall also need the fact that these spherical Artin groups have normal infinite cyclic subgroups \cites{bs72, d72}.
    Spherical Artin groups are known to be linear \cite{cw02}, and hence residually finite.
    There are few examples of non-spherical Artin groups known to be residually finite, see \cite{j22}.
\end{Example}

\section{Homological growth}\label{s:torsiongrowth}

We now compute $b_{i}^{(2)}(G; \Fp)$ and $t_{i}^{(2)}(G)$ for certain groups which act on contractible complexes with strict fundamental domain.

\subsection{A spectral sequence for \texorpdfstring{$b_{i}^{(2)}(G; \Fp)$}{b2(G; Fp)}}\label{s:computation}
The following theorem generalizes \cite{aos21}*{Theorem 1.1}.
We say a group $G$ is $\Fp$-$L^2$-acyclic if $b_{i}^{(2)}(G; \Fp) = 0$ for all $i$ and all residual chains of subgroups.
\begin{Theorem}\label{t:ss}
    Suppose a residually finite group $G$ acts on a contractible complex $Z$ with a strict fundamental domain $(Q,\d Q)$, so that all nontrivial stabilizers are $\Fp$-$L^2$-acyclic.
    Then
    \[
    b_{i}^{(2)}(G; \Fp) = b_{i}(Q, \d Q; \Fp)=\bar b_{i-1}(\d Q; \Fp),
    \]
    where $\bar b_{i-1}(\d Q; \Fp)$ is the reduced $\Fp$-Betti number of $\d Q$.
\end{Theorem}
\begin{proof}
    For each $\gG_{k}$, there is an equivariant homology spectral sequence \cite{b94}*{(7.10), p.~174} converging to $H_{i+j}(\gG_{k}; \Fp)$, with the $E^1$-sheet given by:
    \[
    E^{1}_{i j}= \bigoplus_{\gs \in Q^{(i)}} H_{j}(G^{\gs}; \Fp[G/\gG_{k}]) \implies H_{i+j}(G; \Fp[G/\gG_{k}]) = H_{i+j}(\gG_{k}; \Fp).
    \]
    Let $\gG_{k}^\gs = G^\gs \cap \gG_{k}$.
    The terms $H_{j}(G^{\gs}; \Fp[G/\gG_{k}])$ decompose as
    \[
    H_{j}(G^{\gs}; \Fp[G/\gG_{k}]) = \bigoplus_{\frac{[G:\gG_{k}]}{[G^{\gs}:\gG_{k}^\gs]}} H_{j}(\gG_{k}^\gs; \Fp).
    \]
    After taking Betti numbers and dividing by $[G:\gG_{k}]$, the terms corresponding to nontrivial stabilizers go to zero by assumption, so up to sublinear error terms in the $E^1$-sheet are concentrated in the $j = 0$ row.
    Therefore, again up to sublinear error the spectral sequence degenerates at the $E^2$-sheet.

    Define a coefficient system $V$ on $Q$ which associates the trivial group to a simplex in the singular set $\d Q$ and $\Fp[G/\gG_{k}]$ to the others.
    Since $Q$ is a strict fundamental domain, the chain complex in the $j = 0$ row $E^{1}_{*,0}$ up to sublinear error is $C_\ast(Q; V)$.
    Hence the $E^{\infty}$ sheet is concentrated up to sublinear error in the $j = 0$ row where it equals $H_\ast(Q; V)$.

    There is an exact sequence
    \[
    0 \to C_\ast(\d Q; \Fp[G/\gG_{k}]) \to C_\ast(Q; \Fp[G/\gG_{k}]) \to C_\ast(Q; V) \to 0.
    \]
    So we have $H_\ast(Q; V) = H_\ast(Q, \d Q; \Fp[G/\gG_{k}])$, and taking Betti numbers and dividing by $[G:\gG_{k}]$ gives $b_{i}^{(2)}(G; \Fp) = b_{i}(Q, \d Q; \Fp)$.
    Finally, the strict fundamental domain $Q$ is a retract of $Z$, so it is contractible and therefore $b_{i}(Q, \d Q; \Fp)=\bar b_{i-1}(\d Q; \Fp)$.
\end{proof}
\begin{remark}
    A similar formula computes $b_i^{(2)}(G)$ for groups acting on a contractible complex with strict fundamental domain $(Q,\partial Q)$ and stabilizers either trivial or $L^2$-acyclic (this follows from an equivariant homology spectral sequence, or the computations in \cite{do12}).
    In this case, $G$ does not have to be residually finite.
    In Section \ref{s:l2torsion}, we use the fact that such a $G$ is $L^2$-acyclic if and only if $\d Q$ is $\qq$-acyclic.
\end{remark}

\subsection{Torsion growth}

We begin our inductive approach to computing torsion growth by showing that circles can be rebuilt with precise control over the norms of the rebuilding maps.
This is similar in spirit to \cite{abfg21}*{Lemma 10.10}.
\begin{lemma}\label{l:circle}
    Let $\mathcal{C}$ and $\mathcal{C}'$ be two circles cellulated with $m$ edges and $n = \floor*{m/k}$ edges for some integer $k < m/2$.
    Then there are homotopy inverses $g: \mathcal{C} \rightarrow \mathcal{C}', g': \mathcal{C}' \rightarrow \mathcal{C}$ so that $gg'$ induces the identity on chains and there is a homotopy $\gs$ between $g'g$ and $\id$ so that the norms of $g$, $g'$ and $\gs$ are all bounded above by $2k$.
\end{lemma}
\begin{proof}

    Divide $\mathcal{C}$ into $(n-1)$ intervals of length $k$ and one interval of length $< 2k$.
    The map $g$ collapses each interval to a single edge, and the map $g'$ maps each edge of $C'$ to an interval.
    The composition $g'g: \mathcal{C} \rightarrow \mathcal{C}$ collapses all but one edge in each interval and maps one edge to the entire interval; the homotopy $\gs$ between $g'g$ and the identity sends each vertex $v$ to the interval between $v$ and $g'g(v)$.
    It suffices to compute the norms of $g, g'$ and $\gs$ restricted to a single interval, since chains in disjoint intervals have orthogonal images.
    The entries of the induced maps on chains of $g, g'$ and $\gs$ are all $\pm 1$, therefore the norms are all bounded by $2k$.
\end{proof}

We now can prove our main theorem.
\begin{Theorem}\label{t:main1}
    Suppose a residually finite group $G$ acts on a contractible complex $Z$ with a strict fundamental domain $(Q,\d Q)$, so that all nontrivial stabilizers have normal infinite cyclic subgroups.
    Then
    \[
    t_{i}^{(2)}(G)= \logtor H_{i-1}(\d Q).
    \]
\end{Theorem}
\begin{proof}
    We are given for each nontrivial stabilizer $\{G^\gs\}_{\gs \in \d Q}$ an exact sequence
    \[
    1 \to C^\gs \to G^\gs \to H^\gs \to 1,
    \]
    where $C^\gs$ is infinite cyclic.
    We can choose a model for such a $BG^\gs$ which is an $S^1$-fibration over $BH^{\gs}$.
    Such a $BG^\gs$ admits a cylindrical filtration where
    \[
    (F_{j},E_{j}) = \bigsqcup_{\gt \in {BH^\gs}^{(j)}} S^1 \times (\gt, \partial \gt) .
    \]
    We assume that each fiber circle is cellulated with $1$ vertex.
    For a subgroup $\gG_{k}$ in our chain, let $G^\gs_{k} = G^\gs \cap \gG_{k}$, and $C^\gs_{k} = C^\gs \cap \gG_{k}$.
    Therefore, the cover $BG^\gs_{k}$ also admits a cylindrical filtration where
    \[
    (F_{j},E_{j}) = \bigsqcup_{\frac{[G^\gs: G^\gs_{k}]}{[C^\gs: C^\gs_{k}]}} \bigsqcup_{\gt \in {BH^{\gs}}^{(j)}} S^1 \times (\gt, \partial \gt) .
    \]
    At this point, the norm of the maps $f_{j}$ in the filtration are uniformly bounded independent of $k$ and each circle is cellulated with $[C^\gs: C^\gs_{k}]$ vertices.

    Let $m_{k} = \min_{\gs \in \d Q} \{[ C^\gs: C^\gs_{k}]\}$.
    By recellulating the fiber circles with $\floor*{[ C^\gs: C^\gs_{k}]/m_{k} } $ vertices we can recellulate each $BG^\gs_{k}$ as ${BG^\gs_{k}}'$ so that
    \[
    \frac{1}{2m_{k}}<\frac{\abs{{BG^\gs_{k}}'}}{\abs{BG^\gs_{k}}} \le \frac{1}{m_{k}},
    \]
    where $\abs{BG}$ is the total number of cells in $BG$.

    On each pair $S^1 \times (\gt, \partial \gt)$, by Lemma \ref{l:circle} we have that the norm of each rebuilding map is less than $2m_{k}$.
    Therefore, Lemma \ref{l:control} guarantees that for each $k$, the $\log$ norm of each rebuilt map for ${BG^\gs_{k}}'$ is bounded by $C\log(m_{k})$ for a constant $C$ independent of $k$.

    We now use the Borel construction and a rebuilding procedure as in Lemma \ref{l:borel} to construct a model for $BG$ which admits a cylindrical filtration of the form
    \[
    (F_{j}, E_{j}) = \bigsqcup_{\gs \in \d Q^{(j)}} BG^\gs \times (\gs, \partial \gs); \quad (F_{n+1}, E_{n+1}) = (Q,\d Q).
    \]
    Let $Y$ be the subcomplex of $BG$ constructed at the penultimate stage of this filtration, and let $Y_{k}$ be its lift to $B\gG_{k}$.
    Then $Y_{k}$ admits a cylindrical filtration where
    \[
    (F_{j}, E_{j}) = \bigsqcup_{\gs \in \d Q^{(j)}} \bigsqcup_{\frac{[G: \gG_{k}]}{[G^\gs: G^\gs_{k}]}} BG^\gs_{k} \times (\gs,\partial \gs).
    \]
    Again, at this point, the norm of the maps $f_{j}$ in the filtration are uniformly bounded independent of $k$ and each circle is cellulated with $[C^\gs: C^\gs_{k}]$ vertices.

    If we replace $BG^\gs_{k}$ with ${BG^\gs_{k}}'$, and use the previously rebuilt maps to rebuild $Y_{k}$ to $Y'_{k}$, then the number of cells in $Y'_{k}$ is $\le \frac{\abs{BG}[G:\gG_{k}]}{m_{k}}$.
    Now, we attach $[G: \gG_{k}]$-copies of $Q$ to $Y'_{k}$ to form $B\gG_{k}$.
    The attaching maps $f$ are rebuilt versions of the inclusion maps $\d Q \to Y_{k}$.
    By Lemma \ref{l:control}, we have that
    \[
    \log \norm{\d_{B\gG_{k}}} < D'\log(m_{k})
    \]
    for a constant $D'$ depending only on $BG$.
    Let $C = D'\abs{BG}$.
    We have a short exact sequence of chain complexes
    \[
    0 \to C_{*}(Y'_{k}) \to C_{*}(B\gG_{k}) \to C_{*}(B\gG_{k}, Y'_{k}) \to 0.
    \]
    The relative complex excises: $C_{*}(B\gG_{k}, Y'_{k}) = \oplus_{[G:\gG_{k}]}C_{*}(Q, \d Q)$.
    Since $Q$ is contractible, $H_{i}(Q, \d Q)= \bar H_{i-1}(\d Q)$ and the corresponding long exact sequence looks like
    \[
    \bigoplus_{[G: \gG_{k}]} \bar H_{i}(\d Q) \xrightarrow{f_{*}} H_{i}(Y'_{k}) \to H_{i}(B\gG_{k}) \to \bigoplus_{[G: \gG_{k}]} \bar H_{i-1}(\d Q) \to H_{i-1}(Y'_{k}).
    \]
    Note that
    \begin{align*}
        \logtor \bigoplus_{[G: \gG_{k}]} \bar H_{i-1}(\d Q) &= [G: \gG_{k}] \logtor \bar H_{i-1}(\d Q)), \\
        e\big(\bigoplus_{[G: \gG_{k}]} \bar H_{i-1}(\d Q)\big) &=e(\bar H_{i-1}(\d Q)).
    \end{align*}
    Therefore, applying Lemma~\ref{l:logtor1} to the above sequence and normalizing gives
    \begin{multline*}
        0 \le \logtor \bar H_{i-1}(\d Q) -\frac{\logtor H_{i}(B\gG_{k})}{[G: \gG_{k}]} + \frac{\logtor (H_{i}(Y'_{k})/\im f_{*}) }{[G: \gG_{k}]} \le \\
        \le \frac{\logtor H_{i-1}(Y'_{k})}{[G: \gG_{k}]} + e(\bar H_{i-1}(\d Q))\frac{\rk H_{i}(Y'_{k}) }{[G: \gG_{k}]}.
    \end{multline*}

    By Proposition \ref{p:gabber}, the terms $\logtor H_{i-1}(Y'_{k})$ and $ \rk H_{i}(Y'_{k})$ are bounded by $\frac{C[G:\gG_{k}]}{m_{k}}\log(m_{k})$, so we only need to bound $\logtor H_{i}(Y'_{k})/\im(f_{*})$.
    This is the same as bounding the $\logtor$ of the cokernel of the map
    \[
    f + \partial: \bigoplus_{[G: \gG_{k}]} Z_{i}(\d Q) \oplus C_{i+1}(Y'_{k}) \to Z_{i}(Y'_{k}).
    \]
    Since $Z_{i}(Y'_{k})$ is a direct summand of $C_{i}(Y'_{k})$, this is the same as the logtor of the cokernel of the map
    \[
    f + \partial: \bigoplus_{[G: \gG_{k}]} Z_{i}(\d Q) \oplus C_{i+1}(Y'_{k}) \to C_{i}(Y'_{k}).
    \]
    The rank of the last matrix $f+\partial$ is bounded by the number of cells in $Y'_{k}$.
    Therefore, Proposition \ref{p:gabber} implies that the $\logtor$ is again bounded by $\frac{C[G:\gG_{k}]}{m_{k}} \log(m_{k})$.
    As $k \to \infty$ implies $m_{k} \to \infty$, we are done.
\end{proof}
\begin{Remark}
    In \cite{abfg21}*{Corollary 10.14}, Abert, Bergeron, Fraczyk, and Gaboriau showed that groups with normal $\zz$-subgroups have what they call cheap rebuilding property.
    This is precisely what we used in the above proof: finite index subgroups can be recellulated with sublinear number of cells and subexponentially bounded boundary maps.
    The proof above extends to groups acting on contractible complexes with strict fundamental domain where the stabilizers are either trivial or have their cheap rebuilding property.
\end{Remark}
\begin{Remark}
    If in Theorem~\ref{t:ss} (respectively Theorem~\ref{t:main1}) instead of assuming the existence of a strict fundamental domain, we assume that \emph{all} stabilizers are $\Fp$-$L^2$-acyclic (respectively have normal infinite cyclic subgroups with type $F$ quotient), then the first parts of the arguments show the vanishing of $b_{i}^{(2)}(G; \Fp)$ (respectively $t_{i}^{(2)}(G)$.) The point being that one can still use the Borel construction and rebuilding to construct models for $BG$ and $B\gG_k$; for Theorem~\ref{t:ss} we get a spectral sequence with all the terms of the $E_1$-sheet sublinear, and for Theorem~\ref{t:main1} we get a rebuilt complex with sublinear growth of cells and polynomially bounded norms of boundary maps.
\end{Remark}
\begin{Remark}
    For Theorem \ref{t:ss}, we only need the $G$-complex $Z$ to be $\Fp$-acyclic, and for Theorem \ref{t:main1}, we only need $Z$ to be acyclic.
    The point being that the Borel construction $EG \times_G Z$ can still be used to compute the $\Fp$ or $\zz$-homology of $G$ and $\gG_{k}$.
\end{Remark}

\section{Universal \texorpdfstring{$L^2$}{L2}-torsion}\label{s:l2torsion}

Friedl and L\"uck~\cite{fl17} introduced \emph{the universal $L^{2}$-torsion} $\gt^{(2)}_{u}$ of an $L^{2}$-acyclic complex of finitely generated based free $\zz G$-modules.
It is an element of the \emph{weak} Whitehead group $\Wh^{w}(G)=K_{1}^{w}(G)/\pm G$, where the weak $K_{1}$-group $K_{1}^{w}(G)$ is an abelian group generated by such complexes with the usual relations (we use multiplicative notation):
\[
\left[0 \to \zz G \xrightarrow{id} \zz G \to 0 \right]=1,
\]
and, if
\[
0 \to A_{*} \to B_{*} \to C_{*} \to 0
\]
is a short exact sequence of such complexes, then
\[
[B_{*}]=[A_{*}][C_{*}].
\]
Furthermore, they showed that $K_{1}^{w}(G)$ can be naturally identified with the group generated by square matrices over $\zz G$ which induce weak isomorphisms on $ \ell^2(G)^n$ with the relations:
\[
[AB]=[A][B]=\left[
\begin{pmatrix}
    A & * \\
    0 & B
\end{pmatrix}
\right].
\]

There are two classical generalizations of determinant due to Fuglede and Kadison, and to Dieudonn\'e, which one can try to apply here.
Let $\cn(G)$ be the group von Neumann algebra of $G$-equivariant bounded operators on $\ell^2(G)$.
The Fuglede--Kadison determinant is defined on $\cn(G)$ matrices and takes values in $\rr_{\geq 0}$.
Let $A$ be such a matrix, thought of as $G$-equivariant bounded operator $\ell^2(G)^n \to \ell^2(G)^m$.
Let $F: [0,\infty) \to [0,\infty)$ be the associated spectral density function, defined by
\[
F(\gl)=\sup \{ \dim_G L \mid \norm{A_{| L}} \leq \gl, \ L \text{ is a Hilbert $G$-submodule of $\ell^2(G)^n$ } \}.
\]
The Fuglede--Kadison determinant of $f$, denoted by $\detg f$, is given by
\[
\detg A = \exp \int_{0+}^\infty \ln(\lambda) dF,
\]
with the convention $\exp(-\infty)=0$.
The group $G$ is \emph{of determinant class} if $\detg A \geq 1$ for any $\zz G$ matrix.

The Fuglede--Kadison determinant $\det_{G}$ is not quite a multiplicative map, but it has enough multiplicativity to preserve relations on the matrix generators of $K_{1}^{w}(G)$~\cite{l02}*{Theorem 3.14(1) and (2)}, and if $G$ is of determinant class we avoid division by 0 when taking inverses, thus we have a homomorphism $\detg: \Wh^{w}(G) \to \rr_{> 0}$.
The image of $\gt^{(2)}_{u}$ under the composition $\log \det_{G}$ is the usual $L^{2}$-torsion $\gr^{(2)}$.

The Dieudonn\'e determinant is a homomorphism from the group of invertible matrices over a skew field to the abelianization of the multiplicative group of the skew field.
It is essentially obtained by doing row reduction and taking the (abelianized) product of the diagonal entries of the result.

A skew field arises as follows.
We can embed $\cn(G)$ into the algebra of affiliated operators $U(G)$.
The algebra $U(G)$ can be defined algebraically as an Ore localization of $\cn(G)$ with respect to the set of non-zero divisors ($\cn(G)$ is far from being an integral domain.) Therefore, elements of $U(G)$ can be expressed as $fg^{-1}$ where $f,g \in \cn(G)$ and $g$ is a weak isomorphism $\ell^2(G) \to \ell^2(G)$.
Let $D(G)$ denote the division closure of $\zz G$ inside $U(G)$ (this is the smallest subring of $U(G)$ containing $\zz G$ which is division closed, i.e.\ elements in $D(G)$ which are invertible in $U(G)$ are already invertible in $D(G)$.) By work of Linnell~\cite{l93} a torsion-free group $G$ satisfies the Atiyah conjecture on integrality of $L^2$-Betti numbers \cite{a76}*{p.
72} if and only if $D(G)$ is a skew field.
In this case, the $L^{2}$-Betti numbers of a complex of finitely generated based free $\zz G$-modules can be computed by tensoring with $D(G)$ and taking the dimension over $D(G)$, $\dim_{D(G)}$, of the resulting homology.
Note that if $G$ has torsion, then $\zz G$ has nontrivial zero divisors and cannot embed into a skew field.

So, assuming that $D(G)$ is a skew field, and also that our complex is $L^{2}$-acyclic, we have that tensoring with $D(G)$ produces an acyclic complex over $D(G)$, and we can define the Reidemeister torsion $\gt$ with coefficients in $D(G)$, see \cite{hk21}*{Section 4}.
It is an element of $\Wh(D(G))=K_1(D(G))/\pm G$, which, via the Dieudonn\'e determinant, can be identified with $D(G)^\times_{\text{ab}}/\pm G$.
By~\cite{hk21}*{Theorem 4.13} the natural map $\Wh^{w}(G) \to \Wh(D(G))$ takes $\gt^{(2)}_{u}$ to $\gt$.

If $X$ is a free $L^{2}$-acyclic cocompact $G$-CW-complex, then by taking the universal torsion of the chain complex we obtain $\gt^{(2)}_{u}(X;G)$, If $G$ is torsion-free and satisfies the Atiyah conjecture we also have $\gt(X;G)$.
These are simple equivariant homotopy invariants.
To get homotopy invariance we need to assume vanishing of the ordinary Whitehead group $\Wh(G)$, see \cite{fl17}*{Theorem 2.5(1)}.

Similarly, if $G$ is of determinant class, then we can define $\gr^{(2)}(X;G)$.
However, since in this case $\detg A=1$ for any invertible $\zz G$ matrix, and therefore $\Wh(G) \subset \ker \det_{G}$, $\gr^{(2)}(X;G)$ is an equivariant homotopy invariant.

If $G$ is $L^{2}$-acyclic, and satisfies the appropriate conditions, then by choosing a finite model for the classifying space and taking the universal cover we obtain $\gt^{(2)}_{u}(G)$, $\gt(G)$, and $\gr^{(2)}(G)$.
These invariants have been computed explicitly for certain classes of groups, see \cite{flt19} for a readable summary of previous computations.
\begin{remark}\label{r:cod}
    Conjecturally all groups are of determinant class.
    The class of such groups is known to be very large, it contains all sofic groups~\cite{es05}, so in particular all residually finite groups.
    Also, conjecturally all torsion-free groups satisfy the Atiyah conjecture and have trivial $\Wh(G)$.
    The latter class is also known to be rather large, in particular it contains all CAT(0) groups~\cite{h93}.
    Comparatively, our knowledge about the Atiyah conjecture is rather limited.
\end{remark}

In the formulas below we will use multiplicative notation for $\gt^{(2)}_{u}$ and $\gt$, and additive for $\gr^{(2)}$.
\begin{Theorem}\label{t:l2torsion1}
    Let $G$ be a group acting on a contractible complex $Z$ with strict fundamental domain $(Q,\d Q)$, so that all nontrivial stabilizers are $L^2$-acyclic.
    Assume $G$ and all stabilizers have $\Wh = 1$, and assume that $\d Q$ is $\qq $-acyclic.
    Then
    \[
    \gt^{(2)}_u(G) = \left(\prod_{i \ge 2} \abs{\tor H_{i-1}(\d Q)}^{(-1)^i} \right)\prod_{\gs \in \d Q} \gt^{(2)}_u(G^\gs)^{(-1)^{\dim \gs}}.
    \]
    If in addition $G$ is of determinant class, then
    \[
    \gr^{(2)}(G) = -\sum_{i \ge 1} (-1)^i \logtor{H_{i}(\d Q)} + \sum_{\gs \in \d Q} (-1)^{\dim \gs} \gr^{(2)}(G_\gs).
    \]
\end{Theorem}
\begin{proof}
    Since $\Wh(G) = 0$, we can compute $\gt^{(2)}_u(G)$ by computing $\gt^{(2)}_u$ of a convenient model for $BG$.
    After doing the Borel construction and our rebuilding procedure, as in Lemma \ref{l:borel}, we have a finite model for $BG$ with a cylindrical filtration of the form
    \[
    (F_{j}, E_{j}) = \bigsqcup_{\gs \in \d Q^{(j)}} BG^\gs \times (\gs, \partial \gs); \quad (F_{n+1}, E_{n+1}) = (Q,\d Q),
    \]
    with finite $BG^{\gs}$.

    Lifting this filtration to the universal cover gives a $G$-filtration
    \[
    \emptyset = X_{-1} \subset X_0 \subset X_1 \subset \dots \subset X_{n} \subset X_{n+1} = EG.
    \]
    Note that for $ 0\leq j \leq n $, $X_{j} - X_{j-1}$ is a disjoint union of copies of $EG^{\gs}\times \gs^{\circ}$, $\gs \in \d Q^{(j)}$, and the stabilizer of each copy is a conjugate of $G^{\gs}$.
    Since $\Wh(G^{\gs}) = 0$, we conclude
    \[
    \gt^{(2)}_{u}(X_{j}, X_{j-1}; G)= \prod_{\gs \in \d Q^{(j)}} \gt^{(2)}_{u} (G^{\gs})^{(-1)^{\dim \gs}}.
    \]

    For $j=n+1$, $EG - X_{n}$ is a disjoint union of copies of $Q - \d Q$, freely permuted by $G$.
    Hence,
    \[
    \gt^{(2)}_{u}(EG, X_{n}; G)= \gt^{(2)}_{u} (Q,\d Q; 1) = \prod_{i \ge 2} \abs{\tor H_{i-1}(\d Q)}^{(-1)^i}.
    \]
    The first formula follows from the sum formula applied to short exact sequences:
    \[
    0 \to C_{*}(X_{j-1}) \to C_{*}(X_{j}) \to C_{*}(X_{j},X_{j-1}) \to 0.
    \]
    Applying $\log\detg$ proves the second formula.
\end{proof}

As in Theorem~\ref{t:ss} and Theorem~\ref{t:main1}, if we assume that $G$ acts on a contractible complex $Z$ with \emph{all} cell stabilizers $L^2$-acyclic, then we do not require a strict fundamental domain.
The proof above shows that in this case
\[
\gt^{(2)}_u(G) = \prod_{\gs \in Z/G} \gt^{(2)}_u(G^\gs)^{(-1)^{\dim \gs}}.
\]

This generalizes the fibration formula ~\cite{fl17}*{Theorem 3.11(5)} and leads to a more explicit formula in the situation of Theorem~\ref{t:main1}:
\begin{corollary}\label{c:nz}
    Suppose a group $G$ acts on a contractible complex $Z$ with a strict fundamental domain $(Q,\d Q)$, so that each nontrivial stabilizer $G^{\gs}$ has a normal infinite cyclic subgroup $\zz=\langle s_{\gs} \rangle$.
    Assume $G$ and all stabilizers have $\Wh = 1$, and assume that $\d Q$ is $\qq $-acyclic.
    Then
    \[
    \gt^{(2)}_u(G) = \left(\prod_{i \ge 2} \abs{ \tor H_{i-1}(\d Q)}^{(-1)^i} \right) \prod_{\gs \in Q} i_{*}([s_{\gs}-1])^{(-1)^{\dim \gs}\chi(G^{\gs}/\zz)}.
    \]
    where $i_*$ is induced by the inclusion $i: G^{\gs} \rightarrow G$.
    If in addition $G$ is of determinant class, then
    \[
    \gr^{(2)}(G) = -\sum_{i \ge 1} (-1)^i \logtor{H_{i}(\d Q)}.
    \]
\end{corollary}
\begin{proof}
    The standard cellular chain complex of the universal cover of $S^{1}$ as a $\zz[\zz]$-module is
    \[
    0 \to \zz[s, s^{-1}] \xrightarrow{s-1} \zz[s, s^{-1}] \to 0,
    \]
    hence $\gt_u^{(2)}(\zz) =[s-1]$.
    Therefore, by the fibration formula for $\gt^{(2)}_u$ we have $\gt_u^{(2)}(G^{\gs}) = i_{*}[s_{\gs}-1]^{\chi(G^{\gs}/\zz)}$ and the first formula follows.
    By \cite{l02}*{Example 3.22}, the Fuglede--Kadison determinant of $(s-1)$ is $1$, which implies the second formula.
\end{proof}

The fibration formula implies that if $G$ and $H$ are $L^2$-acyclic groups, then $\gt^{(2)}_{u}(G \times H) = 1$, as $\chi(G) = \chi(H) = 0$.
This leads to a better formula for graph products of $L^2$-acyclic groups which satisfy all the assumptions of Theorem \ref{t:l2torsion1}, and a completely precise formula for RAAG's.
\begin{corollary}\label{c:l2torsion}
    Let $G_{L}$ be a graph product of $L^2$-acyclic groups $G_{v}$ with $\Wh(G_{L}) = \Wh(G_v) = 1$ based on a $\qq $-acyclic flag complex $L$.
    Then
    \[
    \gt^{(2)}_u(G_L) = \left(\prod_{i \ge 2} \abs{\tor H_{i-1}(L)}^{(-1)^i} \right)\prod_{v \in L^{(0)}} \left(i_\ast(\gt^{(2)}_u(G_v)) \right)^{-\rchi(\Lk_L(v))}.
    \]
    If in addition $G$ is of determinant class, then
    \[
    \gr^{(2)}(G_L) = -\sum_{i \ge 1} (-1)^i \logtor{H_{i}(\d Q)} - \sum_{v \in L^{(0)}} \gr^{(2)}(G_v) \rchi(\Lk_L(v)).
    \]
\end{corollary}
\begin{proof}
    If $\gs$ is a simplex in the boundary of the Davis chamber $\d K$, then its stabilizer is the product $\prod_{v \in \min\gs } G_{v}$, so its torsion is trivial unless $\min\gs$ is a single vertex $v$.
    Such simplices correspond to the simplices in $\Lk_{bL}(v)=b\Lk_{L}(v)$ (including the empty simplex), shifted by 1 in dimension.
    This implies the first formula, and taking Fuglede--Kadison determinant gives the second.
\end{proof}

In particular, since RAAG's have $\Wh=1$ and are residually finite, hence of determinant class, we have
\begin{corollary}\label{c:raag}
    Let $A_{L}$ be a RAAG based on a $\qq $-acyclic flag complex $L$.
    Let $S$ be the vertex set of $L$.
    Then
    \begin{align*}
        \gt^{(2)}_u(A_L) &= \left(\prod_{i \ge 2} \abs{\tor H_{i-1}(L)}^{(-1)^i} \right)\prod_{s \in S} \left([s-1]\right)^{-\rchi(\Lk_L(s))}, \\
        \gr^{(2)}(A_L) &= -\sum_{i \ge 1} (-1)^i \logtor{H_{i}(\d Q)} .
    \end{align*}
\end{corollary}

Theorem \ref{t:main1} together with Corollary~\ref{c:nz} show that under hypotheses of the theorem and the corollary the torsion analogue of L\"uck approximation holds.
In particular, it holds for RAAG's.

We now use the fact that RAAG's satisfy the Atiyah conjecture \cite{los12}.
For $L^2$-acyclic groups which satisfy the Atiyah conjecture, Friedl and L\"uck~\cite{fl17} show that $\gt(G)$ determines a formal difference of polytopes in $H_1(G; \rr)$ (they also assumed $\Wh(G) = 1$, this was shown to be unnecessary in \cite{k20}.) Friedl and L\"uck's motivation came from $3$-manifold topology, indeed for fundamental groups of most aspherical $3$-manifold $\d Q$, they showed that this (in this case a single) polytope is dual (up to scaling a by factor of $\frac{1}{2})$ to the unit ball of the Thurston norm on $M$~\cite{fl17}*{Theorem 3.35}.

Here is a rough description of the process.
(See~\cites{fl17, flt19} for details.) Let $G \xrightarrow{p} H_1(G; \zz)/\tor$ be the canonical surjection.
Given any element $f \in \zz G - 0$, we take the convex hull in $H_1(G; \rr)$ of the image of its support under $p$:
\[
P(f)=\conv(p(\supp(f))) \subset H_1(G;\rr).
\]
Polytopes in $H_1(G; \rr)$ up to translation form an abelian cancellative monoid under Minkowski sum (where the identity element is a single point.) Denote by $\cP(G)$ the associated Grothendieck group to this monoid, its elements are formal differences of polytopes.
An elementary observation is that if the ring $\zz G$ has no zero divisors, then $P$ is a homomorphism of monoids, $P(f g)=P(f)+P(g)$.

Friedl and L\"uck show that if $G$ is torsion-free and satisfies the Atiyah conjecture, then $P$ extends to a group homomorphism
\[
P: D(G)^{\times} \to \cP(G).
\]
Since the target is abelian and elements of $G$ are in the kernel, this descends to a map on $\Wh(D(G))$.
It follows from the homomorphism property, that although for general elements of $D(G)^{\times}$, $P$ is hard to compute, for elements of the form $f g^{-1}$ with $f, g \in \zz G$ it is just the formal difference of the convex hulls of the images of the supports of $f$ and $g$ under $p$.

The \emph{$L^2$-torsion polytope} $P(G)$ of $G$ is defined to be $-P(\gt(G))$ (the negative sign is arbitrary, but this makes the $L^2$-torsion polytope into a single polytope for many groups, such as most $3$-manifold groups, see also \cite{k20}.)

Therefore, Corollary~\ref{c:raag} gives the following computation of the $L^2$-torsion polytope for RAAG's.
We identify $H_1(A_L; \rr)$ with $\rr^S$ for $S$ the vertices of $L$.
Under this identification $P(s-1)$ is the unit interval in the $s$-direction, and we note that Minkowski sum of orthogonal intervals is a cube.
\begin{Theorem}\label{t:polytope}
    Let $A_L$ be a RAAG based on a $\qq $-acyclic flag complex $L$.
    Let $S^+$ denote the set of generators of $A_L$ with $\rchi(\Lk_L(s)) > 0$, and $S^-$ the set of generators of $A_L$ with $\rchi(\Lk_L(s))< 0$.

    Then the $L^2$-torsion polytope $P(A_L)$ is the formal difference of two cubes $C^- - C^+$ in $H_1(A_L; \rr) \cong \rr^S$.
    The $C^+$ cube is contained in $\rr^{S^+}$ and for $s \in S^+$ has length in the $s$-direction equal to $\rchi(\Lk_L(s))$.
    The $C^-$ cube is contained in $\rr^{S^-}$ and for $s \in S^-$ has length in the $s$-direction equal to $-\rchi(\Lk_L(s))$.
\end{Theorem}
\begin{remark}
    If $L$ is a graph, then $A_L$ is $L^2$-acyclic if and only if $L$ is a tree.
    In this case, $A_L$ is a fundamental group of a $3$-manifold, and Theorem~\ref{t:polytope} determines the dual of the unit ball of the associated Thurston norm (this also easily follows from Mayer--Vietoris formulas for $\gt$ of an amalgamated product.) In particular, since $\rchi(\Lk(s)) \ge 0$ for all $s$, we see that $P(A_L)$ is a single polytope.
\end{remark}

Given a polytope $P$ in a vector space $V$ and a homomorphism $\phi: V \to \rr$, the thickness of $P$ with respect to $\phi$ is the diameter of its image:
\[
\thi_\phi(P) = \max_{ P} \phi - \min_{P} \phi.
\]
This extends naturally to formal differences of polytopes in $V$ by taking the difference of thicknesses.
As explained in \cite{flt19}*{p.
73}, the results of \cite{fl19} imply that if $G$ is an $L^2$-acyclic type $F$ group satisfying the Atiyah conjecture, then the $L^{2}$-Euler characteristic of the kernel of any epimorphism $\phi: G \to \zz$ is well-defined and given by
\[
\chi^{(2)}(\ker \phi) = - \thi_\phi(P(G)).
\]
Meier, Meinert, and VanWyk \cite{mmv98} determined which characters $A_{L} \to \zz$ have kernels of type $FP(\qq )$.
In particular, we have the following corollary:
\begin{corollary}
    Let $L$ be $\qq$-acyclic and $\phi: A_L \to \zz$ be an epimorphism so that $\phi(s) \ne 0$ for each generator $s \in S$.
    Then the Euler characteristic of the kernel $\chi(\ker \phi)$ is well-defined, equals to $\chi^{(2)}(\ker \phi)$, and given by
    \[
    \chi(\ker \phi) = -\sum_{s \in S} \abs{\phi(s)} \rchi(\Lk_L(s)).
    \]
\end{corollary}
This formula also holds for certain characters which send some generators to $0$, see \cite{mmv98} for the complete picture.

\end{document}